\documentclass[a4paper,11pt]{amsart}  
\usepackage[a4paper,hmargin={3cm,3cm},vmargin={3cm,3cm}]{geometry}

\usepackage[leqno]{amsmath}
\usepackage{amssymb,amsthm}
\usepackage[mathscr]{euscript}
\usepackage{bbm}
\usepackage{dsfont}
\usepackage[all,arc]{xy}
\DeclareMathAlphabet{\mathpzc}{OT1}{pzc}{m}{it}

\theoremstyle{plain}
\newtheorem{theorem}{Theorem}
\newtheorem{lemma}[theorem]{Lemma}
\newtheorem{proposition}[theorem]{Proposition}
\newtheorem{corollary}[theorem]{Corollary}

\theoremstyle{definition}

\theoremstyle{remark}
\newtheorem{remark}[theorem]{Remark}

\renewcommand{\theenumi}{\arabic{enumi}}

\newcommand{\relto}{{\longrightarrow\hspace*{-2.8ex}{\mapstochar}\hspace*{2.6ex}}}
\newcommand{\kto}{\relbar\joinrel\rightharpoonup}
\newcommand{\krelto}{\,{\kto\hspace*{-2.5ex}{\mapstochar}\hspace*{2.6ex}}}
\newcommand{\mate}[1]{\,^\ulcorner\! #1^\urcorner}
\newcommand{\op}{\mathrm{op}}

\newcommand{\catfont}[1]{\mathsf{#1}}

\newcommand{\SET}{\catfont{Set}}
\newcommand{\REL}{\catfont{Rel}}
\newcommand{\MREL}{\catfont{MRel}}
\newcommand{\Coalg}{\catfont{Coalg}}

\newcommand{\monadfont}[1]{\mathbbm{#1}}
\newcommand{\mU}{\monadfont{U}}

\usepackage[hypertexnames=false]{hyperref} 

\begin{document}

\title{On a coalgebraic view on Logic}

\author{Dirk Hofmann}
\author{Manuel A.\ Martins}
\address{Departamento de Matem\'{a}tica\\ Universidade de Aveiro\\3810-193 Aveiro\\ Portugal}
\email{$\{$dirk,martins$\}$@ua.pt}


\thanks{The authors acknowledge partial financial assistance by FEDER funds through COMPETE -- Operational Programme Factors of Competitiveness (Programa Operacional Factores de Competitividade) and by Portuguese funds through the Center for Research and Development in Mathematics and Applications (University of Aveiro) and the Portuguese Foundation for Science and Technology (FCT -- Funda\c{c}\~ao para a Ci\^encia e a Tecnologia), within project PEst-C/MAT/UI4106/2011 with COMPETE number FCOMP-01-0124-FEDER-022690, and the project MONDRIAN (under the contract PTDC/EIA-CCO/108302/2008). The second author also acknowledges partial financial assistance by the project ``Nociones de Completud'', reference FFI2009-09345 (MICINN - Spain).}

\begin{abstract}
In this paper we present methods of transition from one perspective on logic to others, and apply this in particular to obtain a coalgebraic presentation of logic. The central ingredient in this process is to view consequence relations as morphisms in a category.
\end{abstract}

\maketitle

\section*{Introduction}
Logic is a very important concept in several areas of mathematics to rigorously accommodate the principles of sound reasoning. We can consider two main subareas: (1) the theoretical, such as foundations, model theory, philosophy and formal linguistics, and  (2) the applied such as computer science where several logical system have been used to develop formal methods for verification and specification of software systems (see \cite{MGDT05}). For a nice discussion of notions of general abstract logic we refer to \cite{Bez07} (see in particular to \cite{Bez05}). In this paper we wish to participate in this discussion by presenting universal logic as a covariety of coalgebras for an appropriate functor. This is achieved by viewing consequence relations as morphisms in a category.

This work continues the line of research started by A. Palmigiano in \cite{Pal02} that studies the connection of the theory of consequence operators (from a logic perspective as in \cite{FJ96,Woj88}) and the theory of special abstract mathematical structures, namely coalgebras and dialgebras. Palmigiano presented the translation of some basic notions of the theory of consequence operators into notions of the theory of coalgebras, which is closely related to the well known way topological spaces can be considered as coalgebra for the filter functor discussed in \cite{Gum01}. In this paper we present a slighly different construction which is closer to the one for topological spaces (see Remark \ref{rem:ConnToGumm}).

\subsection*{Outline of the paper} We present logic (to be more precise: a consequence relation) on an abstract set rather then a set of formul\ae. We exhibit various ways to encode this structure which have roots in different fields of mathematics, namely topology and (co)algebra. We start by taking a relational approach to the consequence relation; we define appropriated identity and composition to form a monoid that represents the abstract logic. We discuss properties of these structures and of maps between them that can be expressed by simply (in)equalities with the help of suitable defined compositions, which is useful when transporting notions or ideas from one structure to the other since the transition maps preserve both composition and inequalities.

A similar work is carried on with other two perspectives: topological and coalgebraic. In the former, based on the   well known connection between consequence relations and closure operators, we  discuss the meaning of special maps between closure spaces in the context of abstract logic. Relating to the latter, motivated by the fact that topological spaces can be seen as coalgebras for the filter functor \cite{Gum01} and by the work of Palmigiano refereed above, we show how an abstract algebraic logic can be seen as a coalgebra for a natural functor. We develop a detailed analysis
of the maps between coalgebras and their meaning concerning the underlying abstract logic. We close the paper by presenting an elegant and simple proof that the class of coalgebras coming from abstract logics forms a covariety.

Our presentation here rests partially on general results of \cite{SS08}.

\section{A relational view on Logic.}\label{sect:RelView}

A \emph{consequence relation} $\vdash$ on a set $X$ is a relation $\vdash:PX\relto X$ between subsets of $X$ and points of $X$ which satisfies
\begin{enumerate}
\item if $x\in A$, then $A\,\vdash\,x$,
\item if $A\,\vdash\,x$ and $A\subseteq B$, then $B\,\vdash\,x$, and
\item if $A\,\vdash\,y$ for all $y\in B$ and $B\,\vdash\,x$, then $A\,\vdash\,x$;
\end{enumerate}
for all $A,B\subseteq X$ and $x\in X$. In other words, one requires the reflexivity, weakening and cut rule but cannot anymore insist on structurality simply because our ``formul\ae'' are now structureless points of an abstract set. Thanks to the second condition above, one can substitute the first one by
\begin{enumerate}\renewcommand{\theenumi}{\arabic{enumi}'}
\item $\{x\}\,\vdash\,x$ for all $x\in X$.
\end{enumerate}
The pair $(X,\vdash)$ one calls an \emph{abstract logic}. Given also a set $Y$ with a consequence relation $\Vdash$ and a map $f:X\to Y$, one says that $f$ is \emph{consequence preserving} whenever $A\vdash x$ implies $f(A)\Vdash f(x)$, for all $A\subseteq X$ and $x\in X$; and $f$ is called \emph{conservative} if $A\vdash x\iff f(A)\Vdash f(x)$.

The axioms defining a consequence relation can be elegantly expressed using the calculus of relations as we explain next. Recall that for relations $r:X\relto Y$ and $s:Y\relto Z$, one calculates the composite relation $s\cdot r:X\relto Z$ as
\[
 x\,(s\cdot r)\,z\iff\exists y\in Y\,(x\,r\,y)\,\&\,(y\,s\,z).
\]
Every function can be seen as a relation, and relational composition is actually function composition if $s$ and $r$ are functions. Since the identity function $1_X:X\to X$ acts as an identity with respect to relational composition, one obtains the category $\REL$ of sets and relations. It is worth noting that $\REL$ is actually an \emph{ordered category} since inclusion defines an order relation on the set $\REL(X,Y)$ of relations from $X$ to $Y$, and composition from either side preserves this order. One has a functor
\[
 \SET\to\REL
\]
which interprets every function as a relation, and also a functor
\[
 \SET^\op\to\REL
\]
which takes every function $f:X\to Y$ to its ``inverse image relation'' $f^\circ:Y\relto X$ defined by $y\,f^\circ x$ whenever $f(x)=y$. It is worth noting that $f$ and $f^\circ$ form an adjunction $f\dashv f^\circ$ in $\REL$, meaning that one has the inequalities
\begin{align*}
 1_X\subseteq f^\circ\cdot f &&\text{and}&& f\cdot f^\circ\subseteq 1_Y.
\end{align*}
This fact is in particular useful since it allows to ``shuffle functions around'' in inequalities involving their composition with relations.

\begin{lemma}\label{lem:shuffleAdj}
Let $f:X\to Y$ be a function and $r:A\relto X$, $s:A\relto Y$, $r':X\relto B$ and $s':Y\relto B$ be relations. Then
\begin{align*}
 f\cdot r\subseteq s &\iff r\subseteq f^\circ\cdot s &\text{and}&&
 r'\subseteq s'\cdot f &\iff  r'\cdot f^\circ\subseteq s'. 
\end{align*}
\end{lemma}
\begin{proof}
If $f\cdot r\subseteq s$, then $r\subseteq f^\circ\cdot f\cdot r\subseteq f^\circ\cdot s$; and from $r\subseteq f^\circ\cdot s$ one gets $f\cdot r\subseteq f\cdot f^\circ\cdot s\subseteq s$. The second equivalence one obtains similarly.
\end{proof}

Every relation $r:X\relto Y$ can be lifted to a relation $\hat{P}r:PX\relto PY$ between the powersets of $X$ and $Y$ via
\[
A\,(\hat{P}r)\,B\text{ whenever }
\forall y\in B\,\exists x\in A\,.\,x\,r\,y,
\]
for all $A\subseteq X$ and $B\subseteq Y$. In the sequel we will write $P:\SET\to\SET$ for the powerset functor which sends each set $X$ to its power set $PX$, and a function $f:X\to Y$ to the ``direct image function'' $Pf:PX\to PY$ where $Pf(A)=\{f(x)\mid x\in A\}$, for every $A\subseteq X$. We note that $P$ is actually part of a monad (see \cite{MS04} for details) $(P,e,m)$ where $e_X:X\to PX,\,x\mapsto\{x\}$ and $m_X:PPX\to PX,\,\mathcal{A}\mapsto\bigcup\mathcal{A}$. This means in particular that $e=(e_X)$ and $m=(m_X)$ are natural transformations, so that for every function $f:X\to Y$ one has $e_Y\cdot f=Pf\cdot e_X$ and $Pf\cdot m_X=m_Y\cdot PPf$. Using Lemma \ref{lem:shuffleAdj}, from the latter one obtains $PPf\cdot m_X^\circ\subseteq m_Y^\circ\cdot Pf$, and it is not hard to see that one even has equality $PPf\cdot m_X^\circ= m_Y^\circ\cdot Pf$ again.

\begin{proposition}[\cite{SS08,Sea05}]
The following assertions hold:
\begin{enumerate}
\item For all relations $r:X\relto Y$ and $s:Y\relto Z$, $\hat{P}(s\cdot r)=(\hat{P} s)\cdot(\hat{P} r)$.
\item For every function $f:X\to Y$, $Pf\subseteq\hat{P} f$ and $(Pf)^\circ\subseteq\hat{P}(f^\circ)$.
\item For every function $f:X\to Y$ and relations $s:Y\relto B$ and $r:A\relto Y$, $\hat{P}(s\cdot f)=\hat{P} s\cdot Pf$ and $\hat{P}(f^\circ\cdot r)=(Pf)^\circ\cdot\hat{P} r$.
\item For every relation $r:X\relto Y$, $e_Y\cdot r\subseteq\hat{P}r\cdot e_X$ and $\hat{P} r\cdot m_X=m_Y\cdot\hat{P}\hat{P} r$.
\end{enumerate}
\end{proposition}

We can rewrite now the axioms of a consequence relation as simple \emph{reflexivity} and \emph{transitivity} conditions:
\begin{align*}
\{x\}\,\vdash\,x &&\text{and}&& (\mathcal{A}\,(\hat{P}\!\vdash)\,A\;\&\;A\,a\,x)\Rightarrow (\bigcup\mathcal{A})\,a\,x,
\end{align*}
for all $\mathcal{A}\in PPX$, $A\in PX$ and $x\in X$. Equivalently, and without referring to points, these conditions read as
\begin{align*}
1_X\subseteq (\vdash\cdot e_X) &&\text{and}&& (\vdash\cdot \hat{P}\!\vdash)\subseteq(\vdash\cdot m_X).
\end{align*}
Furthermore, using Lemma \ref{lem:shuffleAdj}, these conditions become 
\begin{align*}
e_X^\circ\subseteq\,\vdash &&\text{and}&& (\vdash\cdot(\hat{P}\!\vdash)\cdot m_X^\circ)\subseteq\,\vdash.
\end{align*}

In general, for relations $r:PX\relto Y$ and $s:PY\relto Z$, we can think of
\begin{equation}\label{eq:Kleisli}
 s\circ r:=s\cdot(\hat{P}r)\cdot m_X^\circ
\end{equation}
as a kind of composite relation $s\circ r:PX\relto Z$. This composition is associative and has the relations $\Delta_X:PX\relto X$ defined by $A\,\Delta_X\, x\iff x\in A$ as ``weak'' identities since
\begin{align*}
 r\subseteq\Delta_Y\circ r&&\text{and}&&  r\subseteq r\circ\Delta_X.
\end{align*}
One has even equality above if and only if $r:PX\relto Y$ is \emph{monotone}, that is, $A\,r\, y$ and $A\subseteq B$ imply $B\,r\,y$. For every set $X$, $\Delta_X:PX\relto X$ is monotone, and the composite $s\circ r$ of monotone relations is again monotone. Therefore we can form the category $\MREL$ having sets as objects, a morphism $r:X\krelto Y$ in $\MREL$ is a monotone relation $r:PX\relto Y$ whose composite $s\circ r:X\krelto Z$ with $s:Y\krelto Z$ is defined by \eqref{eq:Kleisli}, and $\Delta_X:X\krelto X$ is the identity morphism on $X$. All told:
\begin{proposition}
A relation $\vdash:PX\relto X$ is a consequence relation on $X$ if and only if $\vdash:X\krelto X$ is a monoid in $\MREL$ with unit $\Delta_X\subseteq\,\vdash$ and multiplication $(\vdash\circ\vdash)\,\subseteq\,\vdash$.
\end{proposition}

To every function $f:X\to Y$ we associate monotone relations
\begin{align*}
f_\#:=e_Y^\circ\cdot\hat{P}f:PX\relto Y &&\text{and}&& f^\#:=e_X^\circ\cdot\hat{P}(f^\circ):PY\relto X.
\end{align*}
Note that $(1_X)_\#=\Delta_X=1_X^\#$ and, more generally, $f_\#=\Delta_Y\cdot Pf$ and $f^\#=f^\circ\cdot\Delta_Y$, hence $A\,f_\#\,y\iff y\in f(A)$ and $B\,f^\#\,x\iff x\in f^{-1}(B)$, for all $A\subseteq X$, $B\subseteq Y$, $x\in X$ and $y\in Y$.

\begin{lemma}\label{lem:RulesForSharp}
Let $f:X\to Y$ be a function and $r:PZ\relto Y$ and $s:PY\relto Z$ be monotone relations. Then
\begin{align*}
 f^\#\circ r= f^\circ\cdot r &&\text{and}&& s\circ f_\#=s\cdot Pf,
\end{align*}
\end{lemma}
\begin{proof}
We calculate
\[
f^\#\circ r = f^\circ\cdot\Delta_Y\cdot\hat{P} r\cdot m_Z^\circ
=f^\circ\cdot(\Delta_Y\circ r)=f^\circ\cdot r
\]
and
\begin{multline*}
s\circ f_\#=s\cdot\hat{P}(\Delta_Y\cdot Pf)\cdot m_X^\circ 
=s\cdot\hat{P}\Delta_Y\cdot PPf\cdot m_X^\circ\\
=s\cdot\hat{P}\Delta_Y\cdot m_Y^\circ\cdot Pf
=(s\circ\Delta_Y)\cdot Pf=s\cdot Pf.\qedhere
\end{multline*}
\end{proof}
\begin{corollary}\label{cor:RulesForSharp}
Let $f:X\to Y$ and $g:Y\to Z$ be functions. Then the following assertions hold.
\begin{enumerate}
\item $(g\cdot f)_\#=g_\#\circ f_\#$ and $(g\cdot f)^\#=f^\#\circ g^\#$.
\item\label{stat2} $f_\#\dashv f^\#$ in $\MREL$, that is, $\Delta_X\subseteq f^\#\circ f_\#$ and $f_\#\circ f^\#\subseteq\Delta_Y$.
\end{enumerate}
\end{corollary}
\begin{proof}
Regarding the first assertions, we calculate
\[
 g_\#\circ f_\#=g_\#\cdot Pf=e_Z^\circ\cdot\hat{P}g\cdot Pf=e_Z^\circ\cdot\hat{P}(g\cdot f)=(g\cdot f)_\#
\]
and
\[
 f^\#\circ g^\#=f^\circ\cdot g^\#=f^\circ\cdot e_Y^\circ\cdot\hat{P}(g^\circ)=e_X^\circ\cdot\hat{P}(f^\circ\cdot g^\circ)=(g\cdot f)^\#.
\]
To see \eqref{stat2}, consider
\[
 f^\#\circ f_\#=f^\#\cdot Pf=e_X^\circ\cdot\hat{P}(f^\circ)\cdot Pf=e_X^\circ\cdot\hat{P}(f^\circ\cdot f)\supseteq e_X^\circ\cdot\hat{P}(1_X)=\Delta_X
\]
and
\[
 f_\#\circ f^\#=\Delta_Y\cdot Pf\cdot\hat{P}(f^\circ\cdot\Delta_Y)\cdot m_Y^\circ
=\Delta_Y\cdot Pf\cdot Pf^\circ\cdot\hat{P}(\Delta_Y)\cdot m_Y^\circ
\subseteq\Delta_Y\cdot\hat{P}(\Delta_Y)\cdot m_Y^\circ
=\Delta_Y\circ\Delta_Y=\Delta_Y.
\]
\end{proof}

Making use of Lemma \ref{lem:RulesForSharp} and of the adjunctions $f\dashv f^\circ$ in $\REL$ and $f_\#\dashv f^\#$ in $\MREL$, one obtains the following characterisations of consequence preserving maps.
\begin{proposition}
Let $f:X\to Y$ be a map between abstract logics $(X,\vdash)$ and $(Y,\Vdash)$. Then
\begin{align*}
\text{$f$ is consequence preserving}
&\iff (f\cdot \vdash)\,\subseteq\,(\Vdash\,\cdot Pf)\\
&\iff \vdash\,\subseteq (f^\circ\cdot\,\Vdash\,\cdot Pf)
\iff(\vdash\cdot Pf^\circ)\subseteq (f^\circ\cdot\,\Vdash) \\
&\iff \vdash\,\subseteq (f^\#\circ\,\Vdash\,\circ f_\#)
\iff (\vdash\circ f^\#)\subseteq (f^\#\circ\,\Vdash)\\
&\iff (f_\#\circ\vdash)\,\subseteq\,(\Vdash\,\circ f_\#).
\end{align*}
Furthermore, $f$ is conservative if and only if $\vdash\,= (f^\circ\cdot\,\Vdash\,\cdot Pf)$ if and only if $\vdash\, = (f^\#\circ\,\Vdash\,\circ f_\#)$.
\end{proposition}

A relation $r:X\relto Y$ is essentially the same thing as a function $\mate{r}:Y\to PX$, via $\mate{r}(y)=\{x\in X\mid x\,r\,y\}$ and $x\,r\,y\iff x\in\mate{r}(y)$. Therefore a relation $r:PX\relto Y$ corresponds to both a mapping
\[
 \mathcal{C}(r):PX\to PY,\,A\mapsto \{y\in Y\mid A\,r\,y\}
\]
and a mapping 
\[
 \mathcal{U}(r):Y\to PPX,\,y\mapsto \{A\subseteq X\mid A\,r\,y\}.
\]
Furthermore, $r:PX\relto Y$ is monotone if and only if the map $\mathcal{C}(r):PX\to PY$ is monotone, if and only if the function $\mathcal{U}(r):Y\to PPX$ takes value in the set $UX=\{\mathcal{A}\subseteq PX\mid \mathcal{A}\text{ is up-closed}\}$. 
Here we call a subset $\mathcal{A}\subseteq PX$ \emph{up-closed} if $A\in\mathcal{A}$ and $A\subseteq B$ imply $B\in \mathcal{A}$. In the next two sections we will explore both point of views.

\section{A topological view on Logic.}\label{sect:TopView}

In the last subsection we have seen that every monotone relation $r:PX\relto Y$ corresponds precisely to a monotone mapping $\mathcal{C}(r):PX\to PY$. Moreover, this transition preserves the compositional structure of monotone relations, as we show next.
\begin{proposition}
\begin{enumerate}
\item\label{assert1} $\mathcal{C}(\Delta_X)=1_{PX}$, for every set $X$.
\item Let $r:X\krelto Y$ and $s:Y\krelto Z$ be monotone relations. Then $\mathcal{C}(s\circ r)=\mathcal{C}(s)\cdot\mathcal{C}(r)$.
\item\label{assert3} Let $r,r':X\krelto Y$ be monotone relations. Then $r\subseteq r'$ if and only if $\mathcal{C}(r)\le\mathcal{C}(r')$.
\item $\mathcal{C}(f_\#)=Pf$ and $\mathcal{C}(f^\#)=Qf$, for every map $f:X\to Y$. Here $Qf:PY\to PX,\,B\mapsto f^{-1}(B)$ is the right adjoint of $Pf:PX\to PY$.
\end{enumerate}
\end{proposition}
\begin{proof}
It is useful here to think of a subset $A\subseteq X$ as a monotone relation $\alpha:\varnothing\krelto X$, and with this interpretation $\mathcal{C}(r):PX\to PY$ is given by $\alpha\mapsto r\circ\alpha$. This proves at once the assertions \eqref{assert1}-\eqref{assert3}. With the help of Lemma \ref{lem:RulesForSharp} it also tells us that $\mathcal{C}(f^\#)=Qf$ and therefore, since $f_\#\dashv f^\#$, one gets $\mathcal{C}(f_\#)\dashv Qf$, hence $\mathcal{C}(f_\#)=Pf$.
\end{proof}
\begin{remark}
Formally, the proposition above states that $\mathcal{C}$ is a functor whose domain is $\MREL$ and whose codomain is the category with objects sets and morphisms monotone maps between the corresponding powersets. This is certainly related to the construction in \cite[Section 4]{BMM12}.
\end{remark}

From this it follows at once that consequence relations $\vdash$ on a set $X$ correspond precisely to monotone maps $c:=\mathcal{C}(\vdash):PX\to PX$ satisfying $1_{PX}\le c$ and $c\cdot c\le c$, that is,
\begin{enumerate}
\item $A\subseteq B\Rightarrow c(A)\subseteq c(B)$,
\item\label{extens} $A\subseteq c(A)$,
\item\label{idemp} $c(c(A))\subseteq c(A)$;
\end{enumerate}
for all $A,B\subseteq X$. Note that one actually has equality in \eqref{idemp}, thanks to \eqref{extens}. In generally, a function $c:PX\to PX$ satisfying the conditions above is called a \emph{closure operator}, and the pair $(X,c)$ one calls a \emph{closure space}.

A map $f:X\to Y$ between closure spaces $(X,c)$ and $(Y,d)$ is called \emph{continuous} whenever $f$ preserves closure points in the sense that $f(c(A))\subseteq d(f(A))$, for all $A\subseteq X$; clearly, $f$ is continuous if and only if it is consequence preserving with respect to the corresponding consequence relations. This can be equivalently expressed in the calculus of relations as $Pf\cdot c\le d\cdot Pf$ and, since $Pf\dashv Qf$, continuity of $f$ is also equivalent to $c\le Qf\cdot d\cdot Pf$. Furthermore, a continuous map $f:X\to Y$ between closure spaces $(X,c)$ and $(Y,d)$ is called \emph{initial} whenever $c= Qf\cdot d\cdot Pf$, which corresponds precisely to conservative maps of abstract logics. The connection with topology suggests yet another notion: we call a consequence preserving map $f:X\to Y$ \emph{open} whenever, for all $x\in X$ and $B\subseteq Y$ with $B\Vdash f(x)$, there exists $A\subseteq X$ with $A\vdash x$ and $f(A)\subseteq B$. The designation ``open'' is motivated here by the formal similarity with the convergence description of open maps in topology (see \cite{Mob81}, for instance). This condition translates into the inequality $(f^\#\circ\,\Vdash)\, \subseteq \,(\vdash\, \circ f^\#)$, hence, since $(f^\#\circ\,\Vdash)\,\supseteq\,(\vdash\, \circ f^\#)$ follows from $f$ being consequence preserving, $f$ is open if and only if $f^\#\circ\,\Vdash\, = \,\vdash\, \circ f^\#$. 

\section{A coalgebraic view on Logic.}\label{sect:CoalgView}

In this section we will think of an abstract logic $\vdash$ on $X$ as a mapping
\[\alpha:=\mathcal{U}(\vdash):X\to UX,\]
which brings us in the realm of coalgebras. This treatment of logic is motivated by fact that topological spaces can be seen as coalgebras for the filter functor \cite{Gum01}, and the subsequent article \cite{Pal02} where closure systems are described as coalgebras for the ``contravariant closure system functor''. However, our presentation differs slightly from \cite{Pal02} as we consider the up-set functor $U$ (described below) which, moreover, is covariant. As we will see, the latter is not an essential difference since $Uf:UX\to UY$ has an adjoint $Vf:UY\to UX$, for every function $f:X\to Y$.

We recall that an $F$-\emph{coalgebra} is a pair $(A, \alpha)$ consisting of a set $ A$ and a function $\alpha: A \rightarrow FA$, where $F$ is an endofunctor on $\SET$ (or any other category). Given $F$-coalgebras $(A,\alpha)$ and $(B,\beta)$, a \emph{homomorphism} from $(A,\alpha)$ to  $(B,\beta)$ is a function $h:A\to B$ for which the square
\[
\xymatrix{
A \ar[r]^{\alpha} \ar[d]_{h} & F{A} \ar[d]^{F{h}} \\
B \ar[r]^{\beta} & F{B}
}
\]
commutes. The category of all $F$-coalgebras and homomorphisms is usual denoted by $\Coalg(F)$.

In general, the notion of coalgebra is in fact the dual of the notion of algebra, and therefore part of the  mathematical theory of coalgebras can be obtained, by duality, from notions and properties of algebras, for instance the terminal-coalgebra construction can be obtained following the guidelines of the initial-algebra construction (see \cite[Cor.\ 3.20]{Ad05}). However, since coalgebras have been successfully used as models of transition systems it is worth to study coalgebras by themselves.
An important property, that will be used below, is that the forgetful functor $\Coalg(F)\to\SET$ creates coproducts (see \cite[Prop.\ 4.3]{Ad05}). That is, a coproduct of a coalgebras is obtained by equipping the corresponding coproduct of the carries of the coalgebras with the unique coalgebra structure turning the injections in homomorphisms. 

In the sequel $U$ denotes the up-set functor on $\SET$, where
\[
 UX=\{\mathcal{A}\subseteq PX\mid \mathcal{A}\text{ is up-closed}\},
\]
and a function $f:X\to Y$ is mapped to
\[
 Uf:UX\to UY,\,\mathcal{A}\mapsto\{B\subseteq Y\mid f^{-1}(B)\in\mathcal{A}\}.
\]
We remark that the monotone map $Uf:UX\to UY$ has a left adjoint $Vf:UY\to UX$ defined by $\mathcal{B}\mapsto\{f^{-1}(B)\mid B\in\mathcal{B}\}$. It is useful to note that the functor $U:\SET\to\SET$ comes together with the families of maps $\eta_X:X\to UX$ and $\mu_X:UUX\to UX$ ($X$ is a set) defined by
\begin{align*}
 \eta_X(x)=\{A\subseteq X\mid x\in A\}&&\text{and}&&
 \mu_X(\mathfrak{A})=\{A\subseteq X\mid A^\#\in\mathfrak{A}\},
\end{align*}
where $A^\#=\{\mathcal{A}\in UX\mid A\in \mathcal{A}\}$. In technical terms, the triple $\mU=(U,e,m)$ is a monad \cite{MS04}. 
\begin{remark}\label{rem:MonUviaAdj}
The construction $f\mapsto f^\#$ defines actually a functor $(-)^\#:\SET\to\MREL^\op$ which is left adjoint to $\hom(-,1):\MREL^\op\to\SET$. The monad induced by this adjunction is precisely $\mU$.
\end{remark}

We also recall that the \emph{Kleisli category} $\SET_\mU$ of $\mU$ has sets as objects, and a morphism $\rho$ from $Y$ to $X$ in $\SET_\mU$ is a map $\rho:Y\to UX$. Given also $\sigma:Z\to UY$, their composite $\rho*\sigma$ is defined as $\rho*\sigma:=\mu_X\cdot U\rho\cdot \sigma:Z\to UX$; and $\eta_X:X\to UX$ is the identity morphism on $X$ with respect to this composition. Moreover, we put $\rho\le\rho'$ whenever $\rho(y)\subseteq\rho'(y)$ for all $y\in Y$, where $\rho,\rho':Y\to UX$. Since the above-defined composition preserves this order from either side, $\SET_\mU$ is an ordered category as well. In fact:
\begin{proposition}
$\mathcal{U}:\MREL^\op\to\SET_\mU$ is an equivalence of categories where, moreover, $\mathcal{U}(r)\le\mathcal{U}(r')\iff r\subseteq r'$, for all monotone relations $r,r':PX\relto Y$.
\end{proposition}
\begin{proof}
One indeed verifies $\mathcal{U}(s\circ r)=\mathcal{U}(r)*\mathcal{U}(s)$ for all monotone relations $r:PX\relto Y$ and $PY\relto Z$, as well as $\mathcal{U}(\Delta_X)=\eta_X$.
\end{proof}

\begin{corollary}
A map $\alpha:X\to UX$ comes from a consequence relation $\vdash$ on $X$ if and only if
\begin{align*}
 \eta_X(x)\subseteq \alpha(x)&&\text{and}&&
 \mu_X\cdot U\alpha\cdot\alpha(x)\subseteq\alpha(x),
\end{align*}
for all $x\in X$, that is, $\eta_X\le\alpha$ and $\alpha*\alpha\le\alpha$.
\end{corollary}
As before, the second inequality above is necessarily an equality thanks to the first inequality.

For $f:X\to Y$, we define maps
\begin{align*}
 f_\Diamond:X &\to UY &\text{and} && f^\Diamond:Y &\to UX\\
 x &\mapsto\{B\subseteq Y\mid f(x)\in B\} &&&
 y &\mapsto\{A\subseteq X\mid y\in f(A)\},
\end{align*}
and one has
\begin{align*}
 \mathcal{U}(f_\#)=f^\Diamond && \text{and} &&
 \mathcal{U}(f^\#)=f_\Diamond.
\end{align*}
From Lemma \ref{lem:RulesForSharp} we obtain that, with $\rho:Y\to UZ$ and $\sigma:Z\to UY$,
\begin{align*}
 \rho * f_\Diamond=\rho\cdot f && \text{and} &&
 f^\Diamond *\sigma=Vf\cdot\sigma,
\end{align*}
and therefore (c.f. Corollary \ref{cor:RulesForSharp}) $(g\cdot f)_\Diamond=g_\Diamond * f_\Diamond$ and $(g\cdot f)^\Diamond=f^\Diamond * g^\Diamond$, where $g:Y\to Z$, and $f_\Diamond\dashv f^\Diamond$ in $\SET_\mU$.

Let now $(X,\vdash)$ and $(Y,\Vdash)$ abstract logics with corresponding maps $\alpha:X\to UX$ and $\beta:Y\to UY$. Then $f$ is consequence preserving if and only if
\[
  \alpha * f^\Diamond \le f^\Diamond * \beta, 
\]
which is equivalent to
\[
 f_\Diamond * \alpha \le \beta * f_\Diamond,
\]
and this in turn reduces to $Uf\cdot\alpha\le \beta\cdot f$. Moreover, $f$ is conservative if and only if $\alpha=f^\Diamond * \beta * f_\Diamond$, or, equivalently $\alpha=Vf\cdot\beta\cdot f$. Somehow dually, we say that $f$ is \emph{progressive} if $\beta=f_\Diamond * \alpha * f^\Diamond$. Note that conservative maps are the coalgebra morphisms in the sense of \cite{Pal02}, moreover, every conservative map as well as every progressive map is consequence preserving. Finally, $f$ is open if and only if $Uf\cdot\alpha=\beta\cdot f$, that is, $f$ is a morphism of coalgebras. Also note that every open injection is conservative and every open surjection is progressive.

To finish this section, we apply the internal characterisation above and show that the class of coalgebras induced by an abstract logic is a \emph{covariety}, that is, it is closed under homomorphic images, subcoalgebras and sums.

\begin{lemma}
Let $(X,\alpha)$ and $(Y,\beta)$ be coalgebras and $f:X\to Y$ be a map. Then
\begin{enumerate}
\item\label{state1} $\alpha$ is induced by an abstract logic provided that $f$ is conservative and $\beta$ is induced by an abstract logic.
\item $\beta$ is induced by an abstract logic if $f$ is progressive and $\alpha$ is induced by an abstract logic.
\end{enumerate}
\end{lemma}
\begin{proof}
To see \eqref{state1}, just note that
\begin{align*}
\alpha &=f^\Diamond * \beta * f_\Diamond\supseteq f^\Diamond * \eta_Y * f_\Diamond=f^\Diamond * f_\Diamond\supseteq\eta_X
\intertext{and}
\alpha *\alpha
&=f^\Diamond * \beta * f_\Diamond * f^\Diamond * \beta * f_\Diamond
\subseteq f^\Diamond * \beta * \beta * f_\Diamond
\subseteq f^\Diamond * \beta * f_\Diamond=\alpha.
\end{align*}
The second statement can be proven in a similar way.
\end{proof}

\begin{theorem}
The class of coalgebras induced by an abstract logic is a covariety.
\end{theorem}
\begin{proof}
The previous lemma implies at once that the class of coalgebras induced by an abstract logic is closed under homomorphic images and subcoalgebras. To show closedness under the formation of sums, we note that the sum of a family $(X_i,\vdash_i)_{i\in I}$ of abstract logics can be calculated as the disjoint union $X=\coprod_{i\in I}X_i$, equipped with the consequence relation $\vdash$ defined by $A\vdash x$ whenever $(A\cap X_i)\vdash_i x$, where $x\in X_i$ (see also \cite{MST06}). By definition, every inclusion map $k_i:(X_i,\vdash_i)\hookrightarrow(X,\vdash)$ is open. Hence, $\alpha:=\mathcal{U}(\vdash)$ is a coalgebra structure on $X$ making every $k_i:(X_i,\alpha_i)\hookrightarrow(X,\alpha)$ (where $\alpha_i:=\mathcal{U}(\vdash_i)$) a coalgebra morphism, and this tells us that the coalgebra $(X,\alpha)$ is the sum of $(X_i,\alpha_i)_{i\in I}$.
\end{proof}

\begin{remark}\label{rem:ConnToGumm}
The methods employed in this paper can be applied also in other contexts. For instance, if one starts in Section \ref{sect:RelView} with the ultrafilter monad instead of the powerset monad, then an abstract logic ``becomes'' a topological space (see \cite{Bar70}), and the adjunction of Remark \ref{rem:MonUviaAdj} induces the filter monad. In this technical sense, our presentation of abstract logics as coalgebras is the same as the one for topological spaces. In a similar way, we might use the identity monad, then Section \ref{sect:RelView} talks about preordered sets (i.e.\ sets equipped with a reflexive and transitive but not necessarily anti-symmetric relation) and the adjunction of Remark \ref{rem:MonUviaAdj} induces the powerset monad.
\end{remark}

\def\cprime{$'$}


\end{document}